\documentclass[preprint,10pt]{elsarticle}
\usepackage{amsmath}
\usepackage{float}
\usepackage{algorithm}
\usepackage{algpseudocode}
\usepackage{amsthm}
\usepackage{subfigure}
\usepackage{txfonts}
\usepackage{booktabs}
\usepackage{algorithm}
\usepackage{natbib}
\usepackage{natbib}
\usepackage{appendix}
\usepackage{color}
\usepackage{enumitem}
\allowdisplaybreaks[4]
\biboptions{numbers,sort&compress}
\usepackage[colorlinks=true]{hyperref}
\usepackage[top=1in,bottom=1in,left=1in,right=1in]{geometry}
\numberwithin{figure}{section}
\makeatletter
\@addtoreset{algorithm}{section}
\makeatother
\newtheorem{thm}{Theorem}[section]

\newtheorem{lem}{Lemma}[section]
\newtheorem{exm}{Example}[section]
\newtheorem{asm}{Assumption}[section]
\newtheorem{defin}{Definition}[section]
\numberwithin{equation}{section}
\newtheorem{rem}{Remark}[section]

\newtheorem{remark}{Remark}[section]

\numberwithin{table}{section}
\begin{document}
	\begin{frontmatter}
		\title{On convergence of greedy block nonlinear Kaczmarz methods with momentum\footnote{}}
		\date{today}
		\author{Naiyu Jiang}
		\author{Wendi Bao}
		\ead{baowendi@sina.com}
		\author{Lili Xing}
		\author{Weiguo Li}

		\address{College of Science,
			China University of Petroleum, Qingdao 266580, P.R. China}	
		
		\begin{abstract}
			In this paper, for solving nonlinear systems
			 we propose two pseudoinverse-free greedy block methods with momentum by combining the residual-based weighted nonlinear Kaczmarz and heavy ball methods. Without the full column rank assumptions on Jacobi matrices of nonlinear systems, we provide a thorough convergence analysis, and derive upper bounds for the convergence rates of the new methods. Numerical experiments demonstrate that the proposed methods with momentum are much more
			effective than the existing ones.
		\end{abstract}
		\begin{keyword}
			Nonlinear equations; Kaczmarz method; Block method; Heavy ball method
		\end{keyword}
	\end{frontmatter}
	
	\section{Introduction}\label{sec1}
	In this paper, we consider the nonlinear systems of  equations
	\begin{equation}\label{eq1.1}
		f(x)=0,
	\end{equation}
	where $f(x):\mathcal{D}\subseteq R^n\to R^m$ is a continuously differentiable function and $x\in R^n$ is the unknown vector to be solved. $f(x)$ can be rewrite as $[f_1(x),f_2(x),...,f_m(x)]^T$.  There is a solution $x_*$ such that $f(x_*)=0$.
	
	The Kaczmarz method was developed independently by the Polish mathematician Kaczmarz (1937) \cite{k37} and the Italian mathematician Cimmino (1938) \cite{C38}. It was also called the algebraic reconstruction technique (ART) in \cite{GBH70}. It is an iterative method of rows, which is a classical method to solve the large-scale linear system of equations $Ax=b$.
	It only needs to load one row of matrix into memory at a time, which can reduce the storage space for solving the large linear system. Each equation can be viewed as a hyperplane. The current iteration point is projected onto the hyperplane in sequence. In \cite{Strohmer07}, Strohmer and Vershynin explored the expected linear convergence rate of the randomized Kaczmarz algorithm. 	
	Considering the advantages of this method, many experts and scholars have shown great interest in the Kaczmarz method and proposed a series of variants of the Kaczmarz method. For solving linear systems, the greedy methods and block methods \cite{Gower15, Bai18, Niu20,Tan25} are proposed such as the random Kaczmarz (GRK) method \cite{Bai18} and the greedy block Kaczmarz (GBK) method \cite{Niu20}. In \cite{Gower15}, Gower and Richtárik provided the Gaussian Kaczmarz (GK) method defined by
	\begin{equation}\label{1.3}
		x_{k+1}=x_k-\frac{\eta^T(Ax_k-b)}{\|A^T\eta\|_2^{2}}A^T\eta,
	\end{equation}
	where $\eta$ is a Gaussian vector with mean $0\in \mathbb R^m$ and the covariance matrix $I\in
	\mathbb{R}^{m\times m}$, i.e., $\eta\sim N(0,I)$, and $I$ denotes the identity matrix. The GK method can be regarded as a block Kaczmarz method that writes directly the increment in the form of a linear combination of all columns of $A^T$ at each iteration, Recently, in \cite{CH22} Chen and Huang derived a fast deterministic block Kaczmarz (FDBK) method, where a set $U_k$ is first computed according to the greedy index selection strategy \cite{Bai18} and $\eta_k$ is defined by
	\begin{equation}\label{eqet}
		\eta_k=\sum_{i\in U_{k}}(b_i-A_{i,:}x_k)\mu_i.
	\end{equation}
	Its relaxed version was given by \cite{WCZ22}. To accelerate the convergence, Tan, Guo et al. \cite{Tan25}  proposed the adaptive deterministic block Kaczmarz method with momentum that introduces a momentum term to the Gaussian Kaczmarz method and designs a set of block control sequence indexes to select rows. In addition, these methods \cite{Gower15, CH22, Bai18, WCZ22,Tan25} are pseudoinverse-free.	
	
	For the nonlinear system of equations, Wang and Li et al. \cite{Wang21} proposed the nonlinear randomized Kaczmarz (NRK) method, which is obtained by a linear approximation using the Taylor expansion at $x_k$. After the NRK method was proposed, many variants were developed. Zhang and Wang et al. \cite{Zhang23} presented the maximum residual nonlinear Kaczmarz-type (MRNK) method, which is faster than the NK method. Further, Lv and Li et al. \cite{Lv24} used sketch-project technology to propose a class of pseudoinverse-free greedy block nonlinear Kaczmarz methods. The approaches employ the average technique of the Gaussian Kaczmarz method and integrate it with a greedy strategy, which significantly reduces the amount of computation. The sketch matrix used in this method is the Gaussian vector $\sum_{i\in\tau_k}-f_i(x_k)$, where $\tau_k$ is the index set that satisfies the greedy criterion. Ye and Yin \cite{Ye24} developed the residual-based weighted nonlinear Kaczmarz (RBWNK) method. The iteration formula is the same as the methods proposed by Lv. While the RBWNK method has introduced a parameter $q$ into the sketch matrix and is more general. The sketch matrix of methods proposed by Lv is a special case of the RBWNK method. Liu, Li et al. \cite{Liu24} proposed the nonlinear greedy randomized Kaczmarz method with momentum (NGRKm), in which a momentum term is added to NGRK and significantly speeds up the convergence, especially for solving large-scale nonlinear systems of equations. Numerical experiments show that the convergence rate of the NGRKm method is faster than those of the NRK method and the nonlinear randomized Kaczmarz (NURK) method \cite{Wang21}. We note that the convergence analysis of the above methods is mostly obtained under the condition that the Jacobi matrices are full column rank.
	
	Inspired by \cite{Tan25, Ye24, Liu24}, for solving nonlinear systems of equations we develop the residual-based weighted nonlinear Kaczmarz method with momentum (RBWNK-m) and maximum residual residual-based weighted nonlinear Kaczmarz method with momentum (MRWNK-m). By combining the RBWNK method and the heavy ball method, we obtain the two new pseudoinverse-free block methods to accelerate the existing methods. The convergence analysis is explored without the full column rank assumptions on Jacobi matrices. Numerical experiments confirm that the proposed algorithms outperform the existing ones.

	The rest of this paper is organized as follows. In Section \ref{sec2}, we recall the RBWNK method and the heavy ball method. In Section \ref{sec3}, we propose the RBWNK-m method and the MRWNK-m method and analyze their convergence. We perform some numerical experiments to verify the efficiency of the methods proposed in Section \ref{sec4}. In Section \ref{sec5}, we make some conclusions.
	
	\section{Notations and Preliminaries}\label{sec2}
	The notations in this paper are expressed as follows. For any matrix $A\in R^{m\times n}$, $\| \cdot\|$, $\|\cdot\|_q$, $\sigma_{\min}$, $\sigma_{\max}$ denote the Euclidean norm, the $q$ norm, the maximum and minimum nonzero singular values of a matrix $A$, respectively.  $[m]=\{1,2,...,m\}$. And $\tau_k\subseteq[m]$ is the selected index set at the $k$ iteration. $|\tau_k|$ is the cardinal number of the set $\tau_k$.  Let $f^\prime(x)$ be the Jacobi matrix of $f(x)$ and $f^\prime_i(x)$ be the $i$th row of $f^\prime(x)$.
	\subsection{The residual-based weighted nonlinear Kaczmarz method}\label{subsec1}
	In \cite{Ye24}, Ye and Yin proposed the residual-based weighted nonlinear Kaczmarz method for \eqref{eq1.1} as follows. Based on the block Kaczmarz method, they first select an index set $\tau_k$ and then solve the subsystem
	\begin{equation*}
		f^\prime_{\tau_k}(x_k)(x-x_k)=-f_{\tau_k}(x_k),
	\end{equation*}
	where $f_{\tau_k}$ represents the restriction of $f$ on the index set $\tau_k$. Ye and Yin introduced a weighted vector $\eta_k\in R^{|\tau_k|}$, where $\tau_k$ is the index set at each iteration. Then a general weighted hyperplane is given by
	\begin{equation*}
		\eta_k^Tf^\prime_{\tau_k}(x_k)(x-x_k)=- \eta_k^Tf_{\tau_k}(x_k).
	\end{equation*}
	Let the iteration point $x_k$ project onto this new hyperplane, then the new iteration point is given by
	\begin{equation*}
		x_{k+1}=x_k-\frac{ \eta_k^Tf_{\tau_k}(x_k)}{\|(f^\prime_{\tau_k}(x_k))^T \eta_k\|^2}(f^\prime_{\tau_k}(x_k))^T \eta_k.
	\end{equation*}
	To enhance the significance of entries with larger magnitudes in the residual vector $r(x)=-f(x)$, they introduced a parameter $q$ and proposed a novel weighted parameter $\eta_k$ defined as follows:
	\begin{align}\label{eta}
		\eta_k^{(i)}
		&=\begin{cases}
			f_i^{q-1}(x_k),  & q \text{ is even,}\\
			|f_i^{q-2}(x_k)|f_i(x_k),  & q \text{ is odd.}
		\end{cases}
	\end{align}
	The details of the residual-based weighted nonlinear Kaczmarz method are described in Algorithm \ref{alg1} as follows.
	\begin{algorithm}
		\caption{The residual-based weighted nonlinear Kaczmarz (RBWNK) method}\label{alg1}
		\begin{algorithmic}[1]
			\State \textbf{Input:} The initial estimate $x_0\in R^n$, parameter $q$, $\varepsilon$, $K$
			\State Initialized: $k=1$, $r=-f(x_0)$;
			\While{$\|r\|^2\geq\varepsilon$ and $k\leq K$}
			\State Compute $\delta_k=\frac{1}{2}\left(\frac{\mathop{\max}\limits_{i\in[m]}|f_i(x_k)|^2}{\|f(x_k)\|^2}+\frac{1}{m}\right)$
			\State Determine the index set
			$\tau_k=\left\{i||f_i(x_k)|^2\geq \delta_k\|f(x_k)\|^2\right\}$
			\State Compute
			\begin{align*}
				\eta_k^{(i)}
				&=\begin{cases}
					f_i^{q-1}(x_k),  & q \text{ is even}\\
					|f_i^{q-2}(x_k)|f_i(x_k),  & q \text{ is odd}
				\end{cases},\ \ i\in \tau_k
			\end{align*}
			\State Update $x_{k+1}=x_k-\frac{\eta_k^Tf_{\tau_k}(x_k)}{\|(f^\prime_{\tau_k}(x_k))^T\eta_k\|^2}(f^\prime_{\tau_k}(x_k))^T\eta_k$
			\State Compute $r=-f(x_k)$
			\State $k=k+1$;
			\EndWhile
			\State \textbf{Output} $x_{k+1}$.
		\end{algorithmic}
	\end{algorithm}
	
	\begin{rem}\label{remark:1}
		When $q=2$, we have $\eta_k^{(i)}=f_i(x_k)$ and $\eta_k\in R^{|\tau_k|}$ in the RBWNK method. In \cite{Lv24} the iteration of the NGABK method is $ x_{k+1}=x_k-\frac{ \hat{\eta}_k^Tf(x_k)}{\|(f^\prime(x_k))^T\hat{\eta}_k\|^2}(f^\prime(x_k))^T \hat{\eta}_k$ with $\hat{\eta}_k=\sum_{i\in \tau_k}-f_i(x_k)e_i$. Let $I_k=R^{m\times|\tau_k|}$ be the matrix whose columns consist of all the vectors $e_i\in R^m$ with $i\in\tau_k$, it holds that $\eta_k=-I_k^T\hat{\eta}_k$. Then $\eta_k^Tf_{\tau_k}(x_k)=-\hat{\eta}_k^Tf(x_k)$, $f^\prime
		_{\tau_k}(x_k)^T\eta_k=-f^\prime(x_k)^T\hat{\eta}_k$,  and $\|f^\prime(x_k)^T\eta_k\|^2=\|f_{\tau_k}^\prime(x_k)\hat{\eta}_k\|^2$.  Moreover, the greedy criterion of the NGABK method is the same as the RBWNK method. Then through simple mathematical calculations, we find that when $q=2$  the NGABK method and the RBWNK method are equivalent.
	\end{rem}
	
	\subsection{The heavy ball method}\label{subsec2}
	
	For solving the optimization problem $\mathop{\min}\limits_{x\in R^n}f(x)$, the gradient descent (GD) method is a classical method. The iterative format is as follows:
	\begin{equation*}
		x_{k+1}=x_k-\alpha_kf^\prime_k(x_k),
	\end{equation*}
	where $\alpha_k$ is a positive stepsize.  In order to improve the convergence rate of the GD method, Polyak proposed the gradient descent method with momentum (mGD) \cite{Nicolas20} by incorporating a (heavy ball) momentum term $\omega(x_k-x_{k-1})$. The iterative format is as follows:
	\begin{equation*}
		x_{k+1}=x_k-\alpha_kf^\prime_k(x_k)+\omega(x_k-x_{k-1}),
	\end{equation*}
	In both convergence analysis and numerical experiments, the mGD method significantly improves the convergence behaviors.
	
	\section{Main Results}
	\label{sec3}
			\subsection{Greedy block Kaczmarz methods with momentum}	
	
	 Firstly, by combining the heavy ball method and the RBWNK method, we establish the RBWNK methods with momentum for solving \eqref{eq1.1} with the following iteration scheme
	\begin{equation}\label{eq:diedai}
		x_{k+1}=x_k-\frac{\eta_k^Tf_{\tau_k}(x_k)}{\|(f^\prime_{\tau_k}(x_k))^T\eta_k\|^2}(f^\prime_{\tau_k}(x_k))^T\eta_k+\omega(x_k-x_{k-1}),
	\end{equation}
	where $\eta_k\in R^{|\tau_k|} $ satisfied (\ref{eta}) and  $\omega$ is the momentum parameter. To accelerate the convergence, we let larger entries of the residual vector $r=-f(x_k)$ be selected to participate in each iteration as far as possible. Here we use two greedy criterion to choose the row index. The first type of index set is
	\begin{align*}
		&\delta_k=\frac{1}{2}\left(\frac{\mathop{\max}\limits_{i\in[m]}|f_i(x_k)|^2}{\|f(x_k)\|^2}+\frac{1}{m}\right),\\
		&\tau_k=\left\{i||f_i(x_k)|^2\geq \delta_k\|f(x_k)\|^2\right.\}
	\end{align*}
	For this index set, we assume that the norm squared of $f_i^\prime
	(x_k)$ is approximately $\frac{1}{m}$ times as $f^\prime(x_k)$. Thus we don't need to calculate the entire Jacobi matrix, which can lead to a large amount of computation. Using this greedy criterion,
	the RBWNK method with momentum (RBWNK-m) is described in Algorithm \ref{alg2}.
	\begin{algorithm}
		\caption{The residual-based weighted nonlinear Kaczmarz with momentum (RBWNK-m) method}\label{alg2}
		\begin{algorithmic}[1]
			\State \textbf{Input:} The initial estimate $x_0\in R^n$, parameter $q$, $\varepsilon$, $K$, momentum parameter $\omega$
			\State Initialized: $k=1$, $r=-f(x_0)$, $x_1=x_0$;
			\While{$\|r\|^2\geq\varepsilon$ and $k\leq K$}
			\State Compute $\delta_k=\frac{1}{2}\left(\frac{\mathop{\max}\limits_{i\in[m]}|f_i(x_k)|^2}{\|f(x_k)\|^2}+\frac{1}{m}\right)$
			\State Determine the index set
			$\tau_k=\left\{i||f_i(x_k)|^2\geq \delta_k\|f(x_k)\|^2\right\}$
			\State Compute
			\begin{align*}
				\eta_k^{(i)}
				&=\begin{cases}
					f_i^{q-1}(x_k),  & q \text{ is even}\\
					|f_i^{q-2}(x_k)|f_i(x_k),  & q \text{ is odd}
				\end{cases},\ \ i\in \tau_k
			\end{align*}
			\State Update $x_{k+1}=x_k-\frac{\eta_k^Tf_{\tau_k}(x_k)}{\|(f^\prime_{\tau_k}(x_k))^T\eta_k\|^2}(f^\prime_{\tau_k}(x_k))^T\eta_k+\omega(x_k-x_{k-1})$
			\State Compute $r=-f(x_k)$
			\State $k=k+1$;
			\EndWhile
			\State \textbf{Output} $x_{k+1}$.
		\end{algorithmic}
	\end{algorithm}
	According to the maximum residual rule, we have the second selection for the index set
	\begin{equation*}
		\tau_k=\left\{i||f_i(x_k)|^2\geq \rho\mathop{\max}\limits_{i\in[m]}|f_i(x_k)|^2\right\},
	\end{equation*}
	where $\rho$ is a relaxed parameter. Then based on the maximum residual, we obtain the residual-based weighted nonlinear Kaczmarz method with momentum (MRWNK-m) in Algorithm \ref{alg3}. It is clear that they are a class of pseudoinverse-free greedy block Kaczmarz methods.
	\begin{algorithm}
		\caption{The maximum residual residual-based weighted nonlinear Kaczmarz with momentum (MRWNK-m) method}\label{alg3}
		\begin{algorithmic}[1]
			\State \textbf{Input:} The initial estimate $x_0\in R^n$, parameter $q$, $\varepsilon$, $K$, momentum parameter $\omega$
			\State Initialized: $k=1$, $r=-f(x_0)$, $x_1=x_0$;
			\While{$\|r\|^2\geq\varepsilon$ and $k\leq K$}
			\State Determine the index set
			$\tau_k=\left\{i||f_i(x_k)|^2\geq \rho\mathop{\max}\limits_{i\in[m]}|f_i(x_k)|^2\right\}$
			\State Compute
			\begin{align*}
				\eta_k^{(i)}
				&=\begin{cases}
					f_i^{q-1}(x_k),  & q \text{ is even}\\
					|f_i^{q-2}(x_k)|f_i(x_k),  & q \text{ is odd}
				\end{cases},\ \ i\in \tau_k
			\end{align*}
			\State Update $x_{k+1}=x_k-\frac{\eta_k^Tf_{\tau_k}(x_k)}{\|(f^\prime_{\tau_k}(x_k))^T\eta_k\|^2}(f^\prime_{\tau_k}(x_k))^T\eta_k+\omega(x_k-x_{k-1})$
			\State Compute $r=-f(x_k)$
			\State $k=k+1$;
			\EndWhile
			\State \textbf{Output} $x_{k+1}$.
		\end{algorithmic}
	\end{algorithm}
		\subsection{Convergence analysis}
	Now, we will explore the convergence of the two new methods proposed.
	
	\begin{defin}\label{de3.1}
		A continuously differentiable nonlinear function $f(x):\mathcal{D}\subseteq R^n\to R^m$ satisfies the local tangential cone condition in $\mathcal{D}$ if for every $i\in [m]$ and $\forall x_1, x_2\in \mathcal{D}$, there exists $\xi_i\in[0,\xi)$ satisfying $\xi=\max_{i\in[m]}\xi_i\leq\frac{1}{2}$ such that
			\begin{equation}\label{qiexiangzhui}
				|f_i(x_1)-f_i(x_2)-f^\prime_i(x_1)(x_1-x_2)|\leq\xi_i|f_i(x_1)-f_i(x_2)| \quad1\leq i\leq m.
		\end{equation}
	\end{defin}
	
	\begin{lem}
		{\rm\cite{Xing25}} If $f(x):\mathcal{D}\subseteq R^n\to R^m$ satisfies the local tangential cone condition in $\mathcal{D}$ and $x_*$ is a solution of (\ref{eq1.1}) in $\mathcal{D}$. Then the following statements hold.
		\begin{itemize}
			\setlength{\itemsep}{0.1em}
		\setlength{\parskip}{0.1em}
			\item [(i)] Any other solution $\bar{x}_*$ in $\mathcal{D}$ satisfies
			\begin{equation}\label{x*-bar(x)*}
				x_*-\bar{x}_*\in\mathcal{N}(f^\prime(x_*)),
			\end{equation}
			where $\mathcal{N}(f^\prime(x_*))$ denotes the null space of the matrix $f^\prime(x_*)$.
			\item [(ii)] Let $x^\dagger$ denotes the solution of minimal distance to $x_0$, then $x^\dagger$ is unique and satisfies
			\begin{equation}\label{x+-x0}
				x^\dagger-x_0\in \mathcal{R}(f^\prime(x^\dagger)^T),
			\end{equation}
			where $\mathcal{R}(f^\prime(x^\dagger)^T)$ denotes the column space of matrix $f^\prime(x^\dagger)^T$, which is the orthogonal complement subspace of $\mathcal{N}(f^\prime(x^\dagger)).$
		\end{itemize}
			\end{lem}
\begin{proof}
		\begin{itemize}
	\item [(i)] From  Definition \ref{de3.1} and the triangle inequality, it holds
	\begin{align*}
		\|f(x_1)-f(x_2))\|_{2}
		& = \sum\limits_{i=1}^m |f_{i}(x_1)-f_{i}(x_2)|\\
		& \ge  \sum\limits_{i=1}^m \frac{1}{\xi_i} |f_i(x_1)-f_i(x_2)-\nabla  f_i(x_1)^T({{x}_1}-{x}_2)|\\
		& \ge \frac{1}{\xi}\left(\|{f}'({x}_1)({x}_1 - {x}_2)\|_2 - \|{f}({x}_1)-{f}({x}_2)\|_{2}\right) \\
		&\left(\ \text{or}\  \ge  \frac{1}{\xi}\left(\|{f}({x}_1)-{f}({x}_1)\|_{2}-\|{f}'(x_1)({x}_1 - {x}_2)\|_2\right) \right).
	\end{align*}
	Thus we get $$ (1-\xi ) \|{f}({x}_1)-{f}({x}_2)\|_{2} \le  \|{f}'({x}_1)({x}_1 - {x}_2)\|_2 \le (1+\xi) \|{f}({x}_1)-{f}({x}_2)\|_{2},$$
	It follows immediately from the above inequality that
$$\frac{1}{1+\xi } \|{f}'({x}_*)({x}_* - \bar{x}_{*})\|_2 \le \|{f}({x}_*)-{f}(\bar{x}_{*})\|_2 \le \frac{1}{1-\xi} \|{f}'({x}_*)({x}_* - \bar{x}_{*})\|_2 $$
hold for ${x}_*, \bar{x}_{*} \in \mathcal{D}$. Obviously,  $ {f}({x}_*)={f}(\bar{x}_{*}) =0$ if and only if  $ {x}_{*}-\bar{x}_{*} \in \mathcal{N}({f}'({x}_{*}))$.\\

	\item [(ii)] Asumme that $x^\dag=s^\dag+u^\dag,\ {x}_{0}=s_{0}+u_{0}$ satisfy
$s^\dag, s_{0}\in \mathcal{N}({f}'({x}^\dag)),\  u^\dag, u_{0}\in \mathcal{N}({f}'({x}^\dag))^\perp$.  Let ${x}^\dag -{x}_{0} = {y} +\tilde{{x}} -{x}_{0}  $ with $y=s^\dag-s_{0},\ \tilde{{x}}=u^\dag+ s_{0}$. It is easy to see that ${y}\in \mathcal{N}({f}'({x}^\dag))$ and $\tilde{{x}} - {x}_{0}= u^\dag- u_{0}\in \mathcal{N}({f}'({x}^\dag))^\perp = \mathcal{R}({f}'({x}^\dag)^T)$. Consequently, it follows  that $\|{x}^\dag-{x}_{0}\|_2^2 =\|{y}\|_2^2 +\|\tilde{{x}}-{x}_{0}\|_2^2 $. On the other hand, $\tilde{{x}} = {x}^\dag- {y} $ is a solution of \eqref{eq1.1}) from (\ref{x*-bar(x)*}).  Therefore, $ {x}^\dag $ is the   solution of  minimal distance to ${x}_{0}$ if and only if ${y}={0}$. Furthermore, ${x}^\dag -{x}_{0}=\tilde{{x}} - {x}_{0}\in \mathcal{R}({f}'({x}^\dag)^T) $.
	\end{itemize}
\end{proof}
	\begin{asm}\label{asm1}
		We assume that the following conditions hold:
			\begin{itemize}
					\setlength{\itemsep}{0.1em}
				\setlength{\parskip}{0.1em}
			\item [(i)] The system of nonlinear equation $f(x):\mathcal{D}\subseteq R^n\to R^m$ is derivable in a bounded closed $\mathcal{D}$. And the derivative of $f(x)$ is continuous in $\mathcal{D}$.
		\item [(ii)]  The system of nonlinear equation $f(x)$ satisfies the local tangential cone condition in $\mathcal{D}$.
		\item [(iii)] 
		There exists a $x^\dagger\in \mathcal{D}$ such that $f(x^\dagger)=0$.
	\item [(iv)]  $\mathcal{N}(f^\prime(x^\dagger))\subseteq \mathcal{N}(f^\prime(x))$ for all $x\in\mathcal{D}$.		
					\end{itemize}
	\end{asm}
	\begin{lem}
		Suppose that Assumption \ref{asm1} exists, and the sequence $\{x_k\}$ is generated by the iteration scheme \eqref{eq:diedai}. Then the following statements hold.
			\begin{itemize}
					\setlength{\itemsep}{0.1em}
				\setlength{\parskip}{0.1em}
		\item[(i)]\label{Algorithm not break down} The term  $\|(f^\prime_{\tau_k}(x_k))^T\eta_k\|^2$ in the iterative formula is nonzero, i.e., Algorithm \ref{alg2} and Algorithm \ref{alg3} will not break down.
			\item[(ii)] \label{eq:x_k-b-x+} For $\eta_k$ defined by (\ref{eta}), the following inequality exists
		\begin{equation*}
			\|x_k-\frac{\eta_k^Tf_{\tau_k}(x_k)}{\|f^{\prime}_{\tau_k}(x_k)^T\eta_k\|^2}f^{\prime}_{\tau_k}(x_k)^T\eta_k-x^\dagger\|^2\leq\|x_k-x^\dagger\|^2-(1-2\xi)\frac{|\eta_k^Tf_{\tau_k}(x_k)|^2}{\|f^{\prime}_{\tau_k}(x_k)^T\eta_k\|^2}.
		\end{equation*}
			\item[(iii)] \label{xk-x+ in R(f'(x+))}$x_k-x^\dagger\in R(f^\prime(x^\dagger)^T)$.
			\end{itemize}
		\begin{proof}
			\begin{itemize}
					\setlength{\itemsep}{0.1em}
				\setlength{\parskip}{0.1em}
		\item[(i)] The proof is analogous to the proof of Theorem 1 in \cite{Ye24}, so we omit it.
		\item[(ii)] Note that compared with the iterative formula of RBWNK, the iterative formula \eqref{eq:diedai} has an extra term $\omega(x_k-x_{k-1})$. From Lemma 3.2 in \cite{Ye24} we can easily obtain the conclusion.
		\item[(iii)] From Assumption \ref{asm1} (iv), we have $\mathcal{N}(f^\prime(x^\dagger))\subseteq\mathcal{N}(f^\prime(x_k))$ for all $k=1,2,...$, i.e.,
		\begin{equation}\label{fk}
			R(f^\prime(x_k)^T)\subseteq R(f^\prime(x^\dagger)^T), \ \text{for}\ k=1,2,....
		\end{equation}
		By (\ref{x+-x0}), we can obtain $$x_1-x^\dagger=x_0-x^\dagger\in R(f^\prime(x^\dagger)^T).$$
		Since $x_1=x_0\in\mathcal{D}$ and Assumption \ref{asm1} (iv), from the above relation it results in
		\begin{equation}\label{x2-x+}
			x_2-x^\dagger=x_2-x_1+x_1-x^\dagger\in R(f^\prime(x_1)^T)\cup R(f^\prime(x^\dagger)^T)\subseteq R(f^\prime(x^\dagger)^T).
		\end{equation}
		
		From the iteration scheme \eqref{eq:diedai}, and the relations \eqref{fk} and \eqref{x2-x+}, it follows that
		\begin{equation}
		 x_3-x^\dagger=x_3-x_2+x_2-x^\dagger\in  R(f^\prime(x_2)^T)\cup R(f^\prime(x_1)^T) \cup R(f^\prime(x^\dagger)^T)\subseteq R(f^\prime(x^\dagger)^T),
		 \end{equation} $$x_4-x^\dagger=x_4-x_3+x_3-x^\dagger\in R(f^\prime(x_3)^T)\cup R(f^\prime(x_2)^T)\cup R(f^\prime(x^\dagger)^T)\subseteq R(f^\prime(x^\dagger)^T),$$
	 $$ ...$$
	   $$x_k-x^\dagger=x_k-x_{k-1}+x_{k-1}-x^\dagger\in  R(f^\prime(x_{k-1})^T)\cup R(f^\prime(x_{k-2})^T)\cup R(f^\prime(x^\dagger)^T)\subseteq R(f^\prime(x^\dagger)^T).$$ Thus,
	   the conclusion holds.
		\end{itemize}
		\end{proof}
	\end{lem}

	\begin{lem}
		\label{1+xi^2}
		{\rm\cite{Zhang24}}
		If the nonlinear function f satisfies the local tangential cone condition in $\mathcal{D}$, then for $\forall x_1, x_2 \in D$  and an index subset $\tau\subseteq [m]$, we have
		\begin{equation}\label{eq:block ltcc}
			\|f_\tau(x_1)-f_\tau(x_2)\|^2\geq\frac{1}{(1+\xi)^2}\|f_\tau^\prime(x_1)(x_1-x_2)\|^2.
		\end{equation}
	\end{lem}

	\begin{lem}{\rm \cite{Nicolas20}}
		\label{a_1,a_2}
		Fix $F_1=F_0\geq0$, and let $\{F_k\}_{k\geq0}$ be a sequence of nonnegative real numbers satisfying the relation
		\begin{equation*}
			F_{k+1}\leq a_1F_k+a_2F_{k-1},k\geq 1,
		\end{equation*}
		where $a_2\geq 0$, $a_1+a_2< 1$. Then the sequence satisfies the relation
		\begin{equation*}
			F_{k+1}\leq p^k(1+\gamma)F_0,k\geq 1,
		\end{equation*}
		where $p=\frac{a_1+\sqrt{a_1^2+4a_2}}{2}$, $\gamma=p-a_1$. Moreover, $a_1+a_2\leq p<1$ with equality if and only if  $a_2=0$.
	\end{lem}
	\begin{lem}\label{lem}
		Suppose that Assumption \ref{asm1} holds. Let $\alpha=
		\underset{1\leq i\leq \tau}{max}\underset{x\in  \mathcal{D}}{sup}\|f'_{i}(x)\|^2.$ Then the sequence $\{x_k\}$ generated by the iterative scheme \eqref{eq:diedai} satisfies the following inequality:
		\begin{align}\label{lem1_eq1}
			\|x_{k+1}-x^\dagger\|^2
			&\leq(1+3\omega+2\omega^2)\|x_k-x^\dagger\|^2+(\omega+2\omega^2)\|x_{k-1}-x^\dagger\|^2\nonumber\\
			& -(1-2\xi+3\omega-4\omega\xi)\frac{|\tau_k|^{1-q}}{\alpha}\|f_{\tau_k}(x_k)\|_2^2.
		\end{align}
		where $|\tau_k|$ is the cardinality of the selected block indices $\tau_k$ and $\omega\in[0,1)$ is the momentum parameter.
		\begin{proof}
			\begin{align*}
				\|x_{k+1}-x^\dagger\|^2
				&=\|x_k-\frac{\eta_k^Tf_{\tau_k}(x_k)}{\|f^{\prime}_{\tau_k}(x_k)^T\eta_k\|^2}f^{\prime}_{\tau_k}(x_k)^T\eta_k+\omega(x_k-x_{k-1})-x^\dagger\|^2\\
				&=\|x_k-\frac{\eta_k^Tf_{\tau_k}(x_k)}{\|f^{\prime}_{\tau_k}(x_k)^T\eta_k\|^2}f^{\prime}_{\tau_k}(x_k)^T\eta_k-x^\dagger\|^2+\omega^2\|x_k-x_{k-1}\|^2\\
				&\quad+2\langle x_k-\frac{\eta_k^Tf_{\tau_k}(x_k)}{\|f^{\prime}_{\tau_k}(x_k)^T\eta_k\|^2}f^{\prime}_{\tau_k}(x_k)^T\eta_k-x^\dagger,\omega(x_k-x_{k-1})\rangle\\
				&=A+B+C.
			\end{align*}
			From Lemma \ref{eq:x_k-b-x+} (ii),  we can obtain
			\begin{equation}\label{A}
				A
				=\|x_k-\frac{\eta_k^Tf_{\tau_k}(x_k)}{\|f_{\tau_k}^{\prime}(x_k)^T\eta_k\|^2}f_{\tau_k}^{\prime}(x_k)^T\eta_k-x^\dagger\|^2
				\leq\|x_k-x^\dagger\|^2-(1-2\xi)\frac{|\eta_k^Tf_{\tau_k}(x_k)|^2}{\|f_{\tau_k}^{\prime}(x_k)^T\eta_k\|^2}.
			\end{equation}
			Next we consider the term $B$ and $C$.
			\begin{align}
				B
				&=2\langle x_k-\frac{\eta_k^Tf_{\tau_k}(x_k)}{\|f^{\prime}_{\tau_k}(x_k)^T\eta_k\|^2}f^{\prime}_{\tau_k}(x_k)^T\eta_k-x^\dagger,\omega(x_k-x_{k-1})\rangle\nonumber\\
				&=2\omega\langle x_k-x^\dagger,x_k-x_{k-1}\rangle-2\omega\langle\frac{\eta_k^Tf_{\tau_k}(x_k)}{\|f^{\prime}_{\tau_k}(x_k)^T\eta_k\|^2}f^{\prime}_{\tau_k}(x_k)^T\eta_k,x_k-x_{k-1}\rangle\nonumber\\
				&=2\omega\|x_k-x^\dagger\|^2+2\omega\langle x_k-x^\dagger,x^\dagger-x_{k-1}\rangle+D\nonumber\\
				&=2\omega\|x_k-x^\dagger\|^2+\omega\|x_k-x_{k-1}\|^2-\omega\|x_{k-1}-x^\dagger\|^2-\omega\|x_k-x^\dagger\|^2+D\nonumber\\
				&=\omega\|x_k-x_{k-1}\|^2+\omega\|x_k-x^\dagger\|^2-\omega\|x_{k-1}-x^\dagger\|^2+D,\label{equation B}
			\end{align}
			where $D=-2\omega\langle\frac{\eta_k^Tf_{\tau_k}(x_k)}{\|f^{\prime}_{\tau_k}(x_k)^T\eta_k\|^2}f^{\prime}_{\tau_k}(x_k)^T\eta_k,x_k-x_{k-1}\rangle$.
			\begin{align}
				C
				&=\omega^2\|x_k-x_{k-1}\|^2\nonumber\\
				&\leq2\omega^2\|x_k-x^\dagger\|^2+2\omega^2\|x_{k-1}-x^\dagger\|^2\label{C}
			\end{align}
			Let $P(x_k)=\frac{\eta_k^Tf_{\tau_k}(x_k)}{\|f^{\prime}_{\tau_k}(x_k)^T\eta_k\|^2}f^{\prime}_{\tau_k}(x_k)^T\eta_k$, then
			\begin{align}
				D
				&=-2\omega P(x_k)^T(x_k-x_{k-1})\nonumber\\
				&=-\omega\left((x_k-x_{k-1})^T+P(x_k)^T-(x_k-x_{k-1})^T\right)(x_k-x_{k-1})-\omega P(x_k)^T\left(P(x_k)+x_k-x_{k-1}-P(x_k)\right)\nonumber\\
				&=-\omega\|x_k-x_{k-1}\|^2-\omega\|P(x_k)\|^2+\omega\|x_k-x^\dagger-P(x_k)-(x_{k-1}-x^\dagger)\|^2\nonumber\\
				&\leq\omega\left(-\|x_k-x_{k-1}\|^2-\|P(x_k)\|^2+2\|x_k-x^\dagger-P(x_k)\|^2+2\|x_{k-1}-x^\dagger\|^2\right)\nonumber\\
				&=\omega\left(-\|x_k-x_{k-1}\|^2+2\|x_k-x^\dagger\|^2+\|P(x_k)\|^2-4\langle x_k-x^\dagger,P(x_k)\rangle+2\|x_{k-1}-x^\dagger\|^2\right)\label{eq:D}
			\end{align}
			The first inequality is obtained by $\|a-b\|^2\leq 2\|a\|^2+2\|b\|^2$, where $a,b\in R^n$. We can get\\
			\begin{equation}\label{equation P}
				\|P(x_k)\|^2=\frac{|\eta_k^Tf_{\tau_k}(x_k)|^2}{\|f_{\tau_k}^{\prime}(x_k)^T\eta_k\|^2}=\frac{\|f_{\tau_k}(x_k)\|_{q}^{2q}}{\|f_{\tau_k}^{\prime}(x_k)^T\eta_k\|^2}.
			\end{equation}
			Let $E=-4\langle x_k-x^\dagger,P(x_k)\rangle$. Imitating the proof of the Lemma 2.3 in \cite{Ye24}, we can get
			\begin{equation}\label{equation E}
				E\leq-4(1-\xi)\frac{\|f_{\tau_k}(x_k)\|_q^{2q}}{\|f_{\tau_k}^{\prime}(x_k)^T\eta_k\|^2}.
			\end{equation}
			The detailed proof can be found in the Appendix \ref{appendix A}.
			Substituting (\ref{equation P}) and (\ref{equation E}) into (\ref{eq:D}), we can obtain
			\begin{equation}\label{equation D}
				\begin{aligned}
					D
					&\leq\omega\left(-\|x_k-x_{k-1}\|^2+2\|x_k-x^\dagger\|^2+\frac{\|f_{\tau_k}(x_k)\|_{q}^{2q}}{\|f_{\tau_k}^{\prime}(x_k)^T\eta_k\|^2}-4(1-\xi)\frac{\|f_{\tau_k}(x_k)\|_q^{2q}}{\|f_{\tau_k}^{\prime}(x_k)^T\eta_k\|^2}+2\|x_{k-1}-x^\dagger\|^2\right)\\
					&=\omega\left(2\|x_k-x^\dagger\|^2+2\|x_{k-1}-x^\dagger\|^2-\|x_k-x_{k-1}\|^2-(3-4\xi)\frac{\|f_{\tau_k}(x_k)\|_p^{2p}}{\|f_{\tau_k}^{\prime}(x_k)^T\eta_k\|^2}\right).
				\end{aligned}
			\end{equation}
			
			Substituting (\ref{equation D}) into inequality (\ref{equation B}), we can get
			\begin{equation*}
				B\leq\omega\left(3\|x_k-x^\dagger\|^2+\|x_{k-1}-x^\dagger\|^2-(3-4\xi)\frac{\|f_{\tau_k}(x_k)\|_q^{2q}}{\|f_{\tau_k}^{\prime}(x_k)^T\eta_k\|^2}\right).
			\end{equation*}
			Combine the above inequality and inequalities (\ref{A}) and (\ref{C}), we can obtain
			\begin{align*}
				\|x_{k+1}-x^\dagger\|^2
				&\leq(1+3\omega+2\omega^2)\|x_k-x^\dagger\|^2+(\omega+2\omega^2)\|x_{k-1}-x^\dagger\|^2-(1-2\xi+3\omega-4\omega\xi)\frac{\|f_{\tau_k}(x_k)\|_q^{2q}}{\|f^{\prime}(x_k)^T\eta_k\|^2}.
			\end{align*}
			From the definition of $\eta_k$ and $f^\prime_{\tau_k}(x_k)$ and imitating the proof of the Theorem 3.1 in \cite{Ye24}, we can obtain
			\begin{equation*}
				\frac{\|f_{\tau_k}(x_k)\|_q^{2q}}{\|f_{\tau_k}^{\prime}(x_k)^T\eta_k\|^2}\geq\frac{|\tau_k|^{1-q}}{\alpha}\|f_{\tau_k}(x_k)\|_2^2.
			\end{equation*}
			More details about the above can be found in Appendix \ref{appendix B}.
			
			Thus
			\begin{align*}
				\|x_{k+1}-x^\dagger\|^2
				&\leq(1+3\omega+2\omega^2)\|x_k-x^\dagger\|^2+(\omega+2\omega^2)\|x_{k-1}-x^\dagger\|^2-(1-2\xi+3\omega-4\omega\xi)\frac{|\tau_k|^{1-q}}{\alpha}\|f_{\tau_k}(x_k)\|_2^2.
			\end{align*}
			Thus the conclusion of the lemma holds.
		\end{proof}
	\end{lem}
	\begin{thm}\label{thm1}
		Suppose that Assumption \ref{asm1} is true. Let $\alpha=
		\underset{1\leq i\leq \tau}{max}\underset{x\in  \mathcal{D}}{sup}\|f'_{i}(x)\|^2$ and $\omega\in[0,1)$ be the momentum parameter. Assume that the following expressions
		$a_1=1+3\omega+2\omega^2-(1-2\xi+3\omega-4\omega\xi)\frac{m^{1-q}\sigma^2_{min}(f^\prime(x^\dagger))}{\alpha (1+\xi)^2}$
		 and $a_2=\omega+2\omega^2$ such that $a_1+a_2<1$. Then the sequence $\{x_k\}$ generated by Algorithm \ref{alg2} satisfied
		\begin{equation*}
			\|x_{k+1}-x^\dagger\|^2\leq p_1^k(1+b_1)\|x_0-x^\dagger\|^2,
		\end{equation*}
		where $p_1=\frac{a_1+\sqrt{a_1^2+4a_2}}{2}$, $b_1=p_1-a_1$ and $a_1+a_2\leq p_1<1$.
		\begin{proof}
			By Lemma \ref{1+xi^2} we can obtain
			\begin{align}\label{thm_eq1}
				\|f_{\tau_k}(x_k)\|_2^2
				&=\sum_{i\in\tau_k}|f_i(x_k)|^2\nonumber\\
				&\geq|\tau_k|\delta_k\|f(x_k)-f(x^\dagger)\|^2\nonumber\\
				&\geq\frac{|\tau_k|}{m(1+\xi)^2}\|f^\prime(x^\dagger)(x_k-x^\dagger)\|^2\nonumber\\
				&\geq\frac{|\tau_k|}{m(1+\xi)^2}\sigma^2_{\min}(f^\prime(x^\dagger))\|x_k-x^\dagger\|^2.
			\end{align}
			The last inequality is makes use of  Lemma \ref{xk-x+ in R(f'(x+))} (iii).
			Then substituting (\ref{thm_eq1}) into (\ref{lem1_eq1}), we can get
			\begin{equation*}
				\begin{aligned}
					\|x_{k+1}-x^\dagger\|^2
					&\leq\left(1+3\omega+2\omega^2-(1-2\xi+3\omega-4\omega\xi)\frac{|\tau_k|^{2-q}\sigma^2_{min}(f^\prime(x^\dagger))}{m\alpha (1+\xi)^2}\right)\|x_k-x^\dagger\|^2+(\omega+2\omega^2)\|x_{k-1}-x^\dagger\|^2\\
					&\leq\left(1+3\omega+2\omega^2-(1-2\xi+3\omega-4\omega\xi)\frac{m^{1-q}\sigma^2_{min}(f^\prime(x^\dagger))}{\alpha (1+\xi)^2}\right)\|x_k-x^\dagger\|^2+(\omega+2\omega^2)\|x_{k-1}-x^\dagger\|^2.
				\end{aligned}
			\end{equation*}
			If
			Let
			$a_1=1+3\omega+2\omega^2-(1-2\xi+3\omega-4\omega\xi)\frac{m^{1-q}\sigma^2_{min}(f^\prime(x^\dagger))}{\alpha (1+\xi)^2}$
			 , $a_2=\omega+2\omega^2$ and $F_k=\|x_k-x^\dagger\|^2$. Then
			\begin{equation}\label{eq:F}
				F_{k+1}\leq a_1F_k+a_2F_{k-1},k\geq1.
			\end{equation}
			Since $\omega\in[0,1)$, then $a_2\geq0$. By assumption we get $a_1+a_2<1$. By Lemma \ref{a_1,a_2} and (\ref{eq:F}) we can obtain
			\begin{equation*}
				\|x_{k+1}-x^\dagger\|^2\leq p_1^k(1+b_1)\|x_0-x^\dagger\|^2,
			\end{equation*}
			where $p_1=\frac{a_1+\sqrt{a_1^2+4a_2}}{2}$, $b_1=p_1-a_1$ and $a_1+a_2\leq p_1<1$.
			Then the conclusion of the Theorem \ref{thm1} is proved.
		\end{proof}
	\end{thm}
	\begin{thm}\label{thm2}
		Suppose that Assumption \ref{asm1} is true. Let $\alpha=
		\underset{1\leq i\leq \tau}{max}\underset{x\in  \mathcal{D}}{sup}\|f'_{i}(x)\|^2$ and $\omega\in[0,1)$ be the momentum parameter. Assume the following expressions
		$a_1=1+3\omega+2\omega^2-(1-2\xi+3\omega-4\omega\xi)\frac{m^{1-q}\rho\sigma^2_{min}(f^\prime(x^\dagger))}{\alpha (1+\xi)^2}$
		 and $a_2=\omega+2\omega^2$ satisfy $a_1+a_2<1$. Then the sequence {${x_k}$} generated by Algorithm \ref{alg3} satisfied
		\begin{equation*}
			\|x_{k+1}-x^\dagger\|^2\leq p_2^k(1+b_2)\|x_0-x^\dagger\|^2,
		\end{equation*}
		where $p_2=\frac{a_1+\sqrt{a_1^2+4a_2}}{2}$, $b_2=p_2-a_1$ and $a_1+a_2\leq p_2<1$.
		\begin{proof}
			By Lemma \ref{1+xi^2} we can obtain
			\begin{align}\label{thm_eq2}
				\|f_{\tau_k}(x_k)\|_2^2
				&=\sum_{i\in\tau_k}|f_i(x_k)|^2\nonumber\\
				&\geq|\tau_k|\rho\max_{i\in[m]}|f_i(x_k)|^2\nonumber\\
				&\geq\frac{|\tau_k|\rho}{m}\|f(x_k)-f(x^\dagger)\|^2\nonumber\\
				&\geq\frac{|\tau_k|\rho}{m(1+\xi)^2}\|f^\prime(x^\dagger)(x_k-x^\dagger)\|^2\nonumber\\
				&\geq\frac{|\tau_k|\rho}{m(1+\xi)^2}\sigma^2_{\min}(f^\prime(x^\dagger))\|x_k-x^\dagger\|^2.
			\end{align}
			The last inequality is makes use of Lemma \ref{xk-x+ in R(f'(x+))} (iii).
			Then substituting (\ref{thm_eq2}) into (\ref{lem1_eq1}), we can get
			\begin{equation*}
				\begin{aligned}
					\|x_{k+1}-x^\dagger\|^2
					&\leq\left(1+3\omega+2\omega^2-(1-2\xi+3\omega-4\omega\xi)\frac{|\tau_k|^{2-q}\rho\sigma^2_{min}(f^\prime(x^\dagger))}{m\alpha (1+\xi)^2}\right)\|x_k-x^\dagger\|^2+(\omega+2\omega^2)\|x_{k-1}-x^\dagger\|^2\\
					&\leq\left(1+3\omega+2\omega^2-(1-2\xi+3\omega-4\omega\xi)\frac{m^{1-q}\rho\sigma^2_{min}(f^\prime(x^\dagger))}{\alpha (1+\xi)^2}\right)\|x_k-x^\dagger\|^2+(\omega+2\omega^2)\|x_{k-1}-x^\dagger\|^2,
				\end{aligned}
			\end{equation*}
			then similar to the proof in Theorem \ref{thm1} the conclusion of the Theorem \ref{thm2} is proved.
		\end{proof}
	\end{thm}		 	
	
	 \begin{remark}
	 	If replace the Assumption \ref{asm1} (iv) by $f^\prime(x_*)$ is column full rank matrix. Then $\mathcal{N}(f^\prime(x_*))=\{0\}$. If $\bar{x}_*$ is another solution of (\ref{eq1.1}), from (\ref{x*-bar(x)*}) we have $x_*-\bar{x}_* \in \mathcal{N}(f^\prime(x_*))$. So $x_*=\bar{x}_*$, which implies the solution of (\ref{eq1.1}) is unique. Under this situation, Algorithms \ref{alg2} and \ref{alg3} converge to the unique solution.
	 \end{remark}
	
	\section{Numerical experiments}
	\label{sec4}
	In this section, some numerical experiments are present to verify the effectiveness of the RBWNK-m method and the MRWNK-m method for solving nonlinear systems of equations. We use the number of iteration steps (IT) and the elapsed computing time in seconds (CPU) to measure the effectiveness of the methods. The IT and CPU are obtained from the average value obtained from 10 experiments. The momentum parameters of RBWNK-m and MRWNK-m are selected by an exhaustive method from 0.01 to 0.99 at intervals of 0.01 each time to select the optimal. All experiments are terminated when $r^2=\|-f(x_k)\|^2\leq 10^{-6}$ or IT exceed $10000$. In Remark \ref{remark:1} we know when $q=2$ the RBWNK method is the NGABK method in \cite{Lv24}. Similarly, we can obtain the MRWNK-m method with $q=2$, $\omega=0$ equal to the MRABK method. For convenience, the MRWNK-m method with $\omega=0$ is written as MRWNK. In tables, the "$>$" defines the IT exceeding 10000.
	
	All experiments are performed using MATLAB (version R2021a) on a personal computer with a 2.20 GHz central processing unit  (13th Gen Inte l(R) Core (TM) i9-13900HX), 32.0 GB memory and Windows operating system (64 bit Windows 11).
	
	\begin{exm}
		Singular Broyden problem\cite{Gomes92}\label{ex:Singular Broyden}
				\begin{align*}
			&f_k(x)=\left((3-2x_k)x_k-2x_{k+1}+1\right)=0, &k=1,\\
			&f_k(x)=\left((3-2x_k)x_k-x_{k-1}-2x_{k+1}+1\right)=0, &1<k<n,\\
			&f_k(x)=\left((3-2x_k)x_k-x_{k-1}+1\right)=0, &k=n.
		\end{align*}
	\end{exm}
	In this experiment, we set the initial estimate $x_0=(-0.5,-0.5,...,-0.5)^T\in R^n$ and $n=100,500,1000$, $\omega=0.5$ in both RBWNK-m and MRWNK-m. From Table \ref{tab:singular rho} we find when $\rho=0.1$, $0.2$, $0.3$ the IT and CPU are less. Moreover, when other settings are fixed, IT and CPU increase with $\rho$. IT and CPU are larger for $\rho=0.1$ with some settings. So we choose $\rho=0.2$. Next we study the effect of $q$ on different methods. In Fig.\ref{figure singular q}, we depict the curves of IT versus the parameter $q$ with $n = 100$,  $500$, $1000$ and $q$ ranges from $2$ to $9$, respectively. From Fig.\ref{figure singular q} we can obtain the best choice of the value of $q$ for different methods. We select for each of these methods the $q$ that has relatively little IT for the different settings. Here we study RBWNK and RBWNK-m with $q=4$ and MRWNK and MRWNK-m with $q=2$. To explore the effect of the momentum term on the convergence efficiency, we depict the curves of $r$ versus IT for different methods for $n=100$, $500$, $1000$ in Fig.\ref{figure singular}. We list IT and CPU with the settings of the above in Table \ref{tab:singular}. Fig.\ref{figure singular} and Table \ref{tab:singular} show that the methods with momentum terms have less IT and CPU than those without, which verifies the efficiency of RBWNK-m and MRWNK-m. Especially RBWNK-m dramatically reduces IT and CPU than RBWNK. Although MRWNK and MRWNK-m have the same IT when $n=500$, MRWNK-m has fewer CPU. We can infer that adding momentum terms can speed up convergence.
	
	\begin{table}
		\centering
		\caption{IT of MRWNK for Singular Broyden problem with different $\rho$}
		\label{tab:singular rho}
		\vspace{1mm}
		\begin{tabular}{ccllllllll}
			\toprule
			$\rho$&   $0.1$& $0.2$& $0.3$& $0.4$& $0.5$&$0.6$ & $0.7$& $0.8$&$0.9$\\
			\midrule
			$m=100,q=2$&
			$\mathbf{32}$& $48$& $67$& $165$& $223$& $258$& $435$& $1288$&$1867$
			\\
			$m=100,q=3$& $\mathbf{60}$& $67$& $71$& $176$& $270$& $344$& $480$& $763$&$1648$
			\\
			$m=100,q=4$& $98$& $\mathbf{95}$& $120$& $247$& $283$& $384$& $516$& $968$&$1932$
			\\
			$m=500,q=2$& $33$& $\mathbf{31}$& $44$& $303$& $831$& $1439$& $2911$& $5438$&$>$\\
			$m=500,q=3$& $53$& $\mathbf{49}$& $64$& $589$& $665$& $1278$& $2461$& $5049$&$8300$
			\\
			$m=500,q=4$& $110$& $82$& $\mathbf{73}$& $90$& $541$& $816$& $2097$& $>$&$8308$
			\\
			$m=1000,q=2$& $\mathbf{31}$& $37$& $54$& $67$& $1655$& $2633$& $6909$& $8343$&$>$\\
			$m=1000,q=3$& $61$& $\mathbf{54}$& $57$& $1085$& $2371$& $3101$& $5567$& $>$&$>$\\
			$m=1000,q=4$& $128$& $93$& $\mathbf{76}$& $1370$& $680$& $1588$& $3038$& $7581$&$>$\\ \bottomrule\end{tabular}
		
	\end{table}
	\begin{table}
		\centering
		\caption{IT and CPU of different methods for Singular Broyden problem with corresponding $q$}
		\label{tab:singular}
		\vspace{1mm}
		\begin{tabular}{cclll}
			\toprule
			Method&  $n$& $100$& $500$&$1000$\\
			\midrule
			RBWNK&
			IT & $592$
			& $2651$
			&$6050$
			\\
			($q=4$)&CPU & $0.0021$& $0.0249$&$0.092$\\
			MRWNK&IT
			& $48$ & $31$ &$37$ \\
			($q=2$, $\rho=0.2$)&CPU & $0.0004$& $0.0061$ &$0.0180$ \\
			RBWNK-m&IT
			& $86$& $82$
			&$912$
			\\
			($q=4$, $\omega=0.5$)&CPU & $0.0003$& $0.0067$&$0.0259$\\
			MRWNK-m&IT
			& $\mathbf{23}$& $\mathbf{31}$&$\mathbf{30}$\\
			($q=2$, $\rho=0.2$, $\omega=0.5$)&CPU & $\mathbf{0.0001}$& $\mathbf{0.0028}$&$\mathbf{0.0073}$\\\bottomrule\end{tabular}
		
	\end{table}
	
	\begin{figure}[!t]
		\centering
		\subfigure[$n=100$]{
			\includegraphics[scale=0.35]{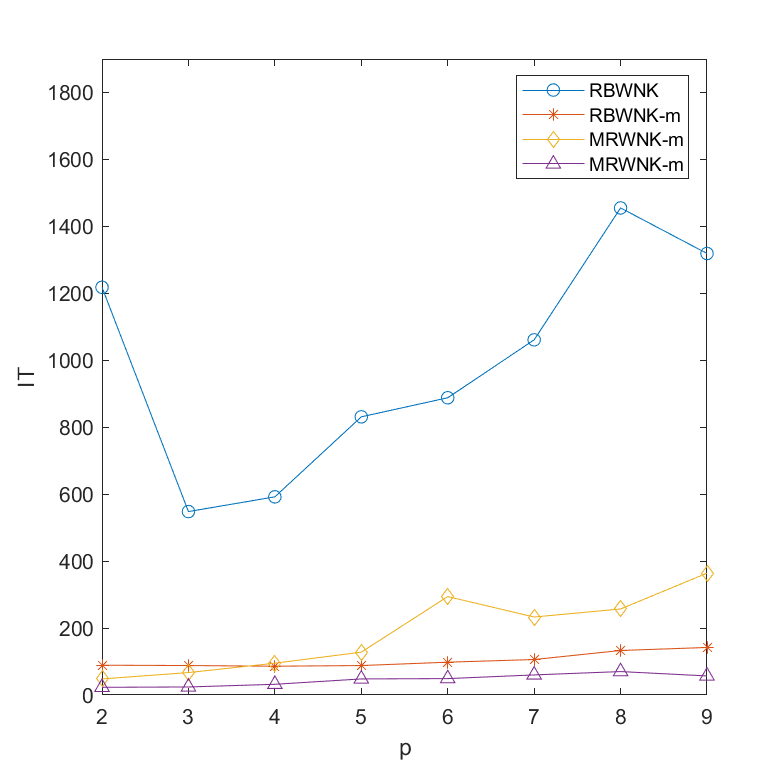}}
		\subfigure[$n=500$]{
			\includegraphics[scale=0.35]{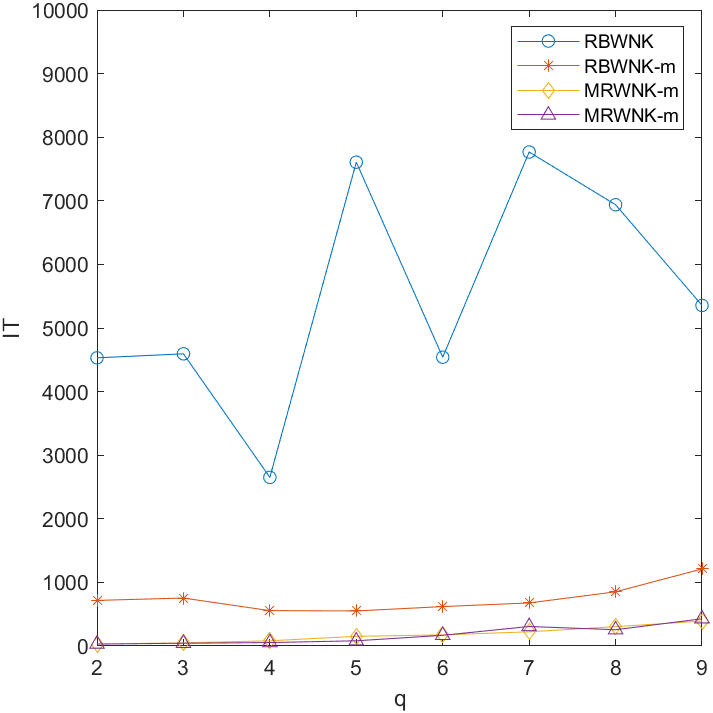}}
		\subfigure[$n=1000$]{	\includegraphics[scale=0.35]{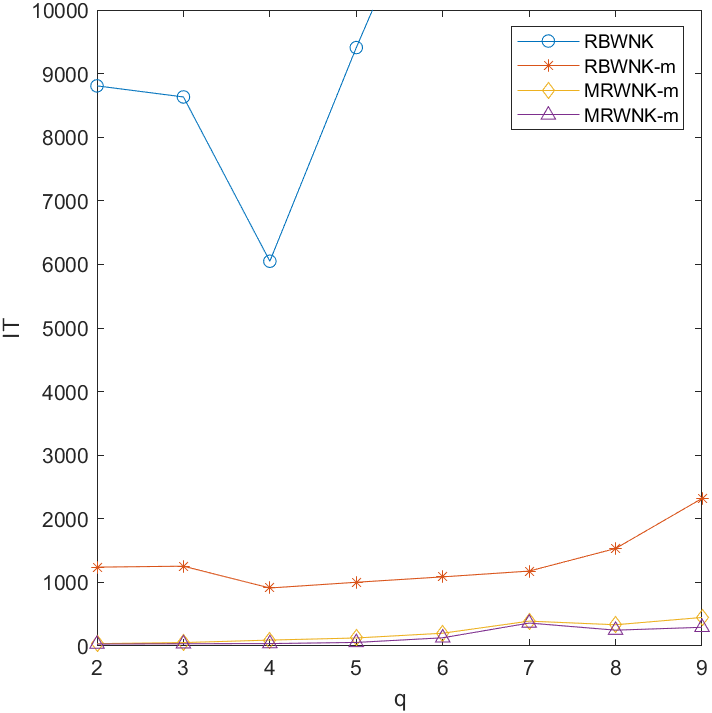}}\\
		\caption[d]{IT of the four methods versus $q$ for Singular Broyden problem with $n=100$, $500$, $1000$}
		\label{figure singular q}
	\end{figure}
	\begin{figure}[!t]
		\centering
		\subfigure[$n=100$]{
			\includegraphics[scale=0.4]{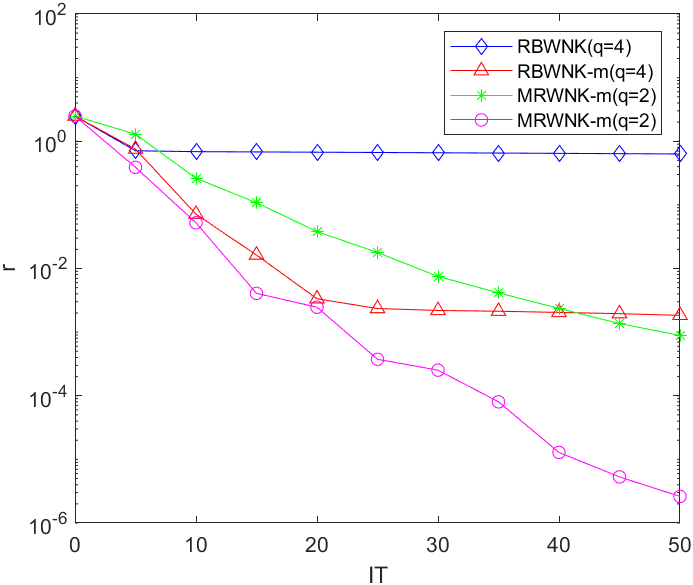}}
		\subfigure[$n=500$]{
			\includegraphics[scale=0.4]{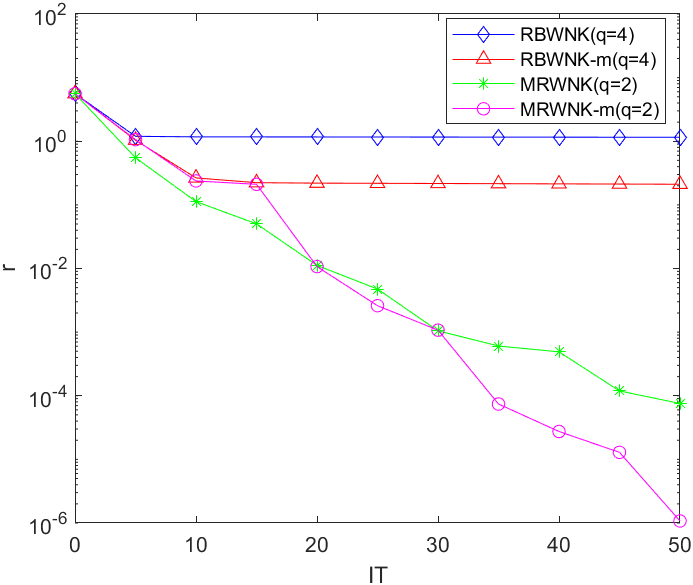}}
		\subfigure[$n=1000$]{	\includegraphics[scale=0.4]{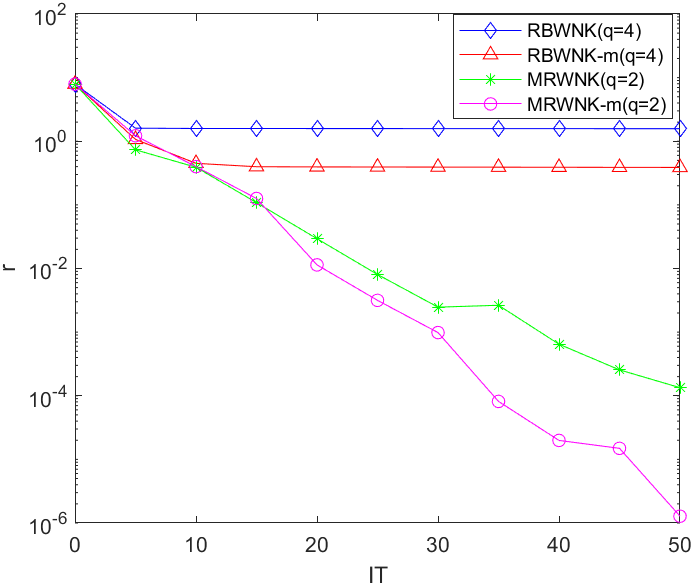}}\\
		\caption{The results of the Singular Broyden problem  $n=100$, $500$, $1000$}
		\label{figure singular}
	\end{figure}

	\begin{exm}
		H-equation \cite{Wang21}\label{ex:H}
		\begin{equation*}
			f_i(x)=x_i-\left(1-\frac{c}{2n}\sum^N_{j=1}\frac{\mu _ix_j}{\mu _i+\mu _j}\right)=0, i=1,2,...n.
		\end{equation*}
		where $\mu _i=(i-1/2)/n$.
	\end{exm}
	\vspace{-3mm}
	In this experiment, we set the initial estimate $x_0=(0,0,...,0)^T\in R^n$ and $n=100$, $500$, $1000$, $\omega=0.1$ in both RBWNK-m and MRWNK-m. From Table \ref{tab:H rho}, we find that in most settings MRWNK with $\rho=0.1$ has the best efficiency. In Fig.\ref{figure H q}, we depict the curves of IT versus the parameter $q$ with $q$ ranging from $2$ to $9$, respectively. In Fig.\ref{figure H q} we observe that the IT for all these methods increases with the value of $q$, so the best choice of $q$ is $2$. So we study the four methods with $q=2$. In Fig.\ref{figure H} and Table \ref{tab:H} we also observe that RBWNK-m and MRWNK-m converge faster than RBWNK and MRWNK, and MRWNK-m has the least IT and CPU.
	
	\begin{table}
		\centering
		\caption{IT of MRWNK for H-equation with different $\rho$}
		\vspace{1mm}
		\begin{tabular}{ccllllllll}
			\toprule
			$\rho$&   $0.1$& $0.2$& $0.3$& $0.4$& $0.5$&$0.6$ & $0.7$& $0.8$&$0.9$\\
			\midrule
			$m=200,q=2$&
			$\mathbf{21}$& $23$& $23$& $35$& $50$& $51$& $76$& $108$&$231$
			\\
			$m=200,q=3$& $39$& $\mathbf{30}$& $36$& $46$& $58$& $75$& $103$& $162$&$300$
			\\
			$m=200,q=4$& $\mathbf{37}$& $49$& $44$& $46$& $60$& $77$& $106$& $157$&$301$
			\\
			$m=400,q=2$& $\mathbf{24}$& $26$& $27$& $39$& $55$& $64$& $89$& $146$&$312$
			\\
			$m=400,q=3$
			& $41$& $\mathbf{36}$& $42$& $53$& $69$& $91$& $126$& $199$&$398$
			\\
			$m=400,q=4$& $61$& $60$& $\mathbf{46}$& $53$& $69$& $92$& $128$& $199$&$398$
			\\
			$m=800,q=2$& $\mathbf{25}$& $27$& $28$& $40$& $58$& $75$& $89$& $152$&$314$
			\\
			$m=800,q=3$& $44$& $\mathbf{37}$& $44$& $57$& $74$& $99$& $136$& $213$&$433$
			\\
			$m=800,q=4$& $70$& $67$& $\mathbf{51}$& $56$& $74$& $99$& $136$& $215$&$433$
			\\ \bottomrule \end{tabular}
		\label{tab:H rho}
	\end{table}
	\begin{table}
		\centering
		\caption{IT and CPU of different methods for H-equation with corresponding $q$}
		\label{tab:H}
		\vspace{1mm}
		\begin{tabular}{cclll}
			\toprule
			Method&  $n$& $100$& $500$&$1000$\\
			\midrule
			RBWNK&
			IT & $53$& $53$&$53$\\
			($q=2$)&CPU & $0.0245$& $0.6989$&$1.8334$\\
			MRWNK&IT
			& $21$& $24$&$25$\\
			($q=2$, $\rho=0.2$)&CPU & $0.0156$& $0.4432$&$1.2540$\\
			RBWNK-m&IT
			& $42$& $43$&$45$\\
			($q=2$, $\omega=0.5$)&CPU & $0.0172$& $0.4629$&$1.3726$\\
			MRWNK-m&IT
			& $\mathbf{19}$& $\mathbf{21}$&$\mathbf{22}$\\
			($q=2$, $\rho=0.2$, $\omega=0.5$)&CPU & $\mathbf{0.0142}$& $\mathbf{0.3779}$&$\mathbf{1.0776}$\\\bottomrule\end{tabular}
	\end{table}

	\begin{table}
	\centering
	\caption{IT of MRWNK for nondquar problem with different $\rho$}
	\label{tab:nondquar rho}
	\vspace{1mm}
	\begin{tabular}{ccllllllll}
		\toprule
		$\rho$&   $0.1$& $0.2$& $0.3$& $0.4$& $0.5$&$0.6$ & $0.7$& $0.8$&$0.9$\\
		\midrule
		$m=200,q=2$&
		$2081$& $1471$& $\mathbf{1214}$& $1246$& $1298$& $1307$& $1330$& $1358$ & $1392$
		\\
		$m=200,q=3$& $1207$& $\mathbf{1181}$& $1186$& $1248$& $1302$& $1317$& $1340$& $1363$&$1389$
		\\
		$m=200,q=4$& $1070$& $1152$& $1161$& $1255$& $1313$& $1335$& $1349$& $1363$&$1398$
		\\
		$m=400,q=2$& $4371$& $3190$& $\mathbf{2545}$& $2583$& $2649$& $2740$& $2766$& $2822$&$2879$
		\\
		$m=400,q=3$
		& $2523$& $2465$& $\mathbf{2428}$& $2627$& $2604$& $2786$& $2800$& $2857$&$2884$
		\\
		$m=400,q=4$& $\mathbf{2308}$& $2396$& $2384$& $2600$& $2653$& $2790$& $2809$& $2838$&$2899$
		\\
		$m=800,q=2$& $10000$& $6496$& $\mathbf{5233}$ & $5331$& $5565$& $5610$& $5755$& $5850$&$5990$
		\\
		$m=800,q=3$& $5191$& $5211$& $\mathbf{5151}$& $5444$& $5488$& $5754$& $5803$& $5897$&$6016$
		\\
		$m=800,q=4$& $\mathbf{4944}$& $5112$& $5157$& $5463$& $5651$& $5850$& $5768$& $5903$&$6002$
		\\ \bottomrule \end{tabular}
	\end{table}
    
	\begin{figure}[!t]
		\centering
		\subfigure[$n=100$]{
			\includegraphics[scale=0.35]{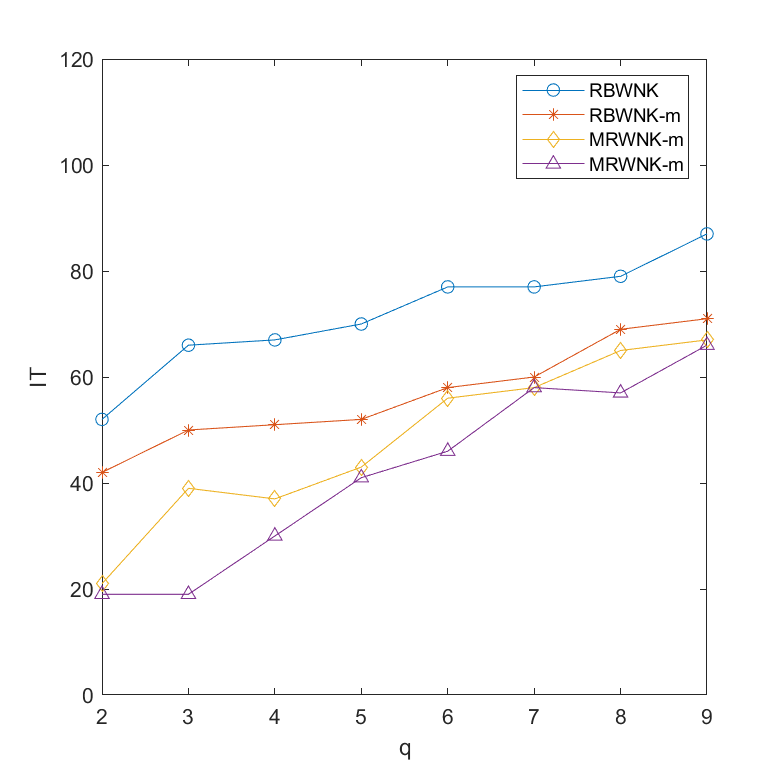}}
		\subfigure[$n=500$]{
			\includegraphics[scale=0.35]{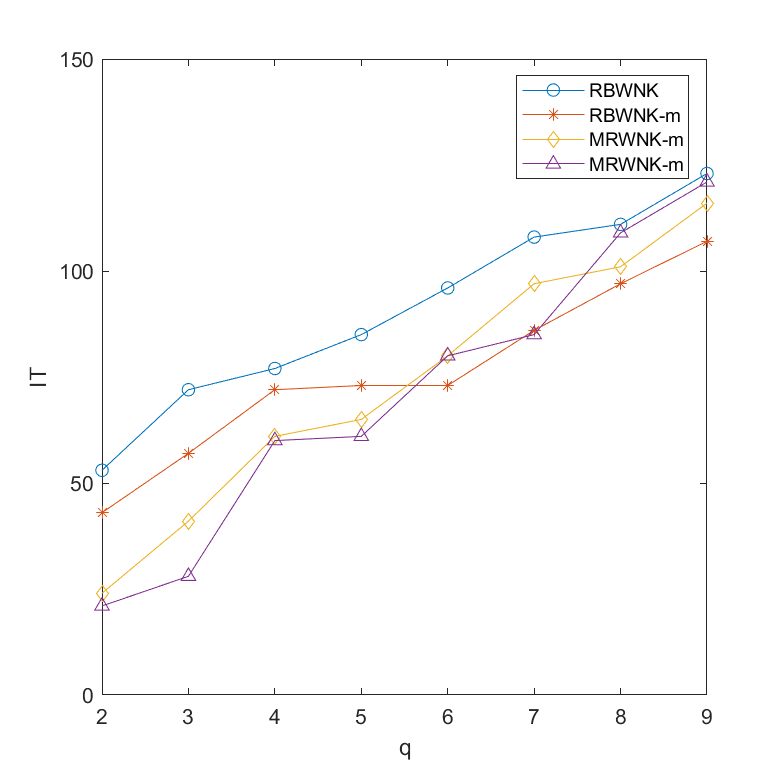}}
		\subfigure[$n=1000$]{
			\includegraphics[scale=0.35]{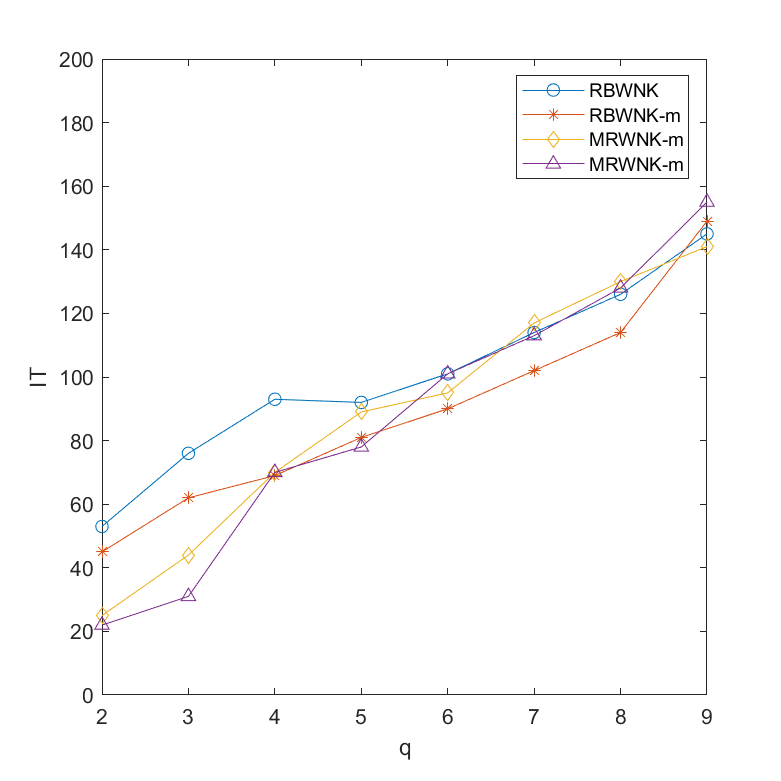}}\\
		\caption[d]{IT of  the four methods versus $q$ for H-equation with $m=n=100$, $500$, $1000$}
		\label{figure H q}
	\end{figure}
	
	\begin{figure}[!t]
		\centering
		\subfigure[$n=100$]{
			\includegraphics[scale=0.4]{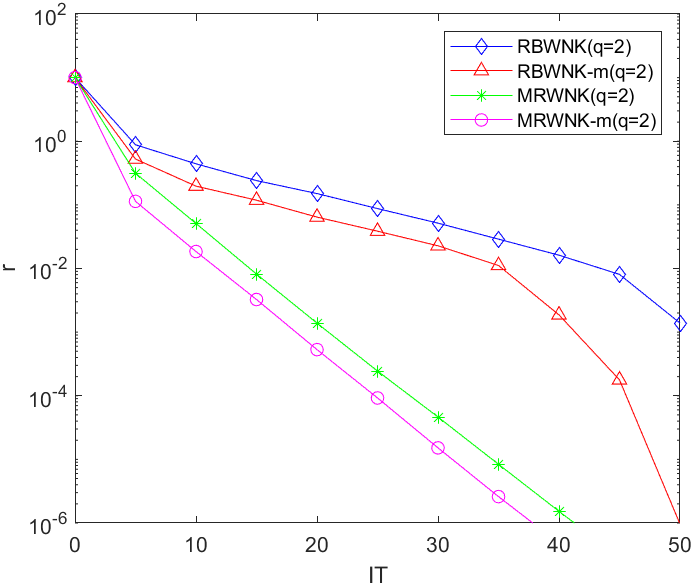}}
		\subfigure[$n=500$]{
			\includegraphics[scale=0.4]{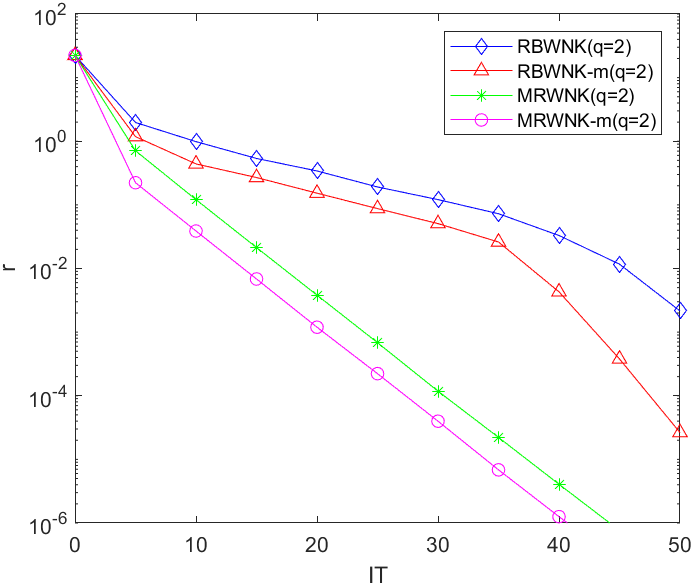}}
		\subfigure[$n=1000$]{	\includegraphics[scale=0.4]{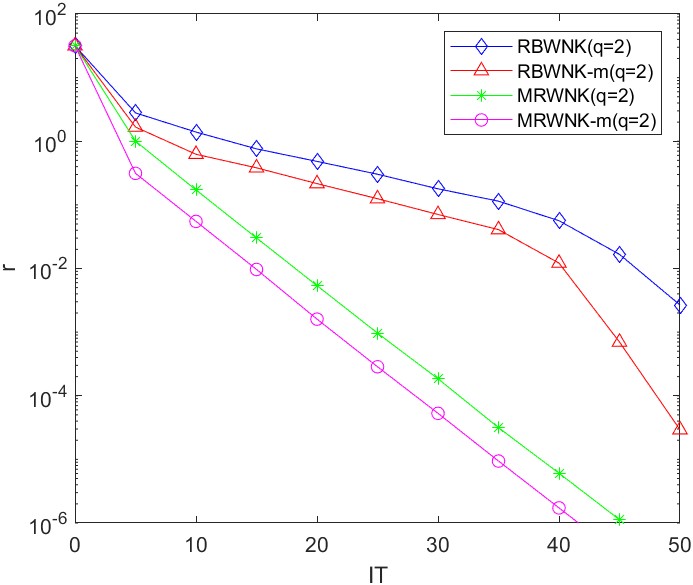}}\\
		\caption{The results of the H-equation with $m=n=100$, $500$, $1000$}
		\label{figure H}
	\end{figure}
	
	\begin{exm}
		NONDQUAR problem \cite{Luksan18}\label{ex:nondquar}
				\begin{align*}
			&f_k(x)=(0.5x_k-3)x_k+x_{k+1}-1=0, &k=1,\\
			&f_k(x)=(0.5x_k-3)x_k+x_{k-1}+x_{k+1}-1=0, &1<k<n,\\
			&f_k(x)=(0.5x_k-3)x_k+x_{k-1}-1=0, &k=n.
		\end{align*}
	\end{exm}
	We set the initial estimate $x_0=(-0.5,-0.5,...,-0.5)^T\in R^n$ and $n=200$, $400$, $800$, $\omega=0.7$ in RBWNK-m and $\omega=0.79$ in MRWNK-m. From Table \ref{tab:nondquar rho}, we let $\rho=0.3$ . From Fig.\ref{figure nondquar q} we observe that the value of $q$ has little effect on the convergence of RBWNK and RBWNK-m, while MRWNK and MRWNK-m are sensitive to the value of $q$. Here we study the four methods with $q=4$. Table \ref{tab:nondquar} shows that MRWNK has fewer IT than RBWNK, but MRWNK has longer CPU. When $n=200$ and $400$ although MRWNK-m has fewer IT than RBWNK-m, MRWNK-m has longer CPU. But when $n=800$ MRWNK-m has less CPU. Further more, RBWNK-m and MRWNK-m have fewer IT and CPU than RBWNK and MRWNK, which further validates the momentum term efficiency.

	\begin{table}
		\centering
		\caption{IT and CPU of different methods for NONDQUAR problem corresponding $q$}
		\label{tab:nondquar}
		\vspace{1mm}
		\begin{tabular}{cclll}
			\toprule
			Method&  $n$&$200$&$400$&$800$\\
			\midrule
			RBWNK&
			IT &$1368$&$2814$&$5856$\\
			($q=4$)&CPU &$0.0213$&$0.2147$&$1.7476$\\
			MRWNK&IT
			&$1161$&$2384$&$5157$\\
			($q=4$, $\rho=0.3$)&CPU &$0.0559$&$0.5588$&$3.3595$\\
			RBWNK-m&IT
			&$801$&$1752$&$4081$\\
			($q=4$, $\omega=0.7$)&CPU &$\mathbf{0.0117}$&$\mathbf{0.1038}$&$0.6773$\\
			MRWNK-m&IT
			&$\mathbf{685}$&$\mathbf{1131}$&$\mathbf{2357}$\\
			($q=4$, $\rho=0.3$, $\omega=0.79$)&CPU &$0.0131$&$0.1083$&$\mathbf{0.5303}$\\\bottomrule\end{tabular}
	\end{table}
	\begin{figure}[!t]
		\centering
		\subfigure[$n=200$]{
			\includegraphics[scale=0.35]{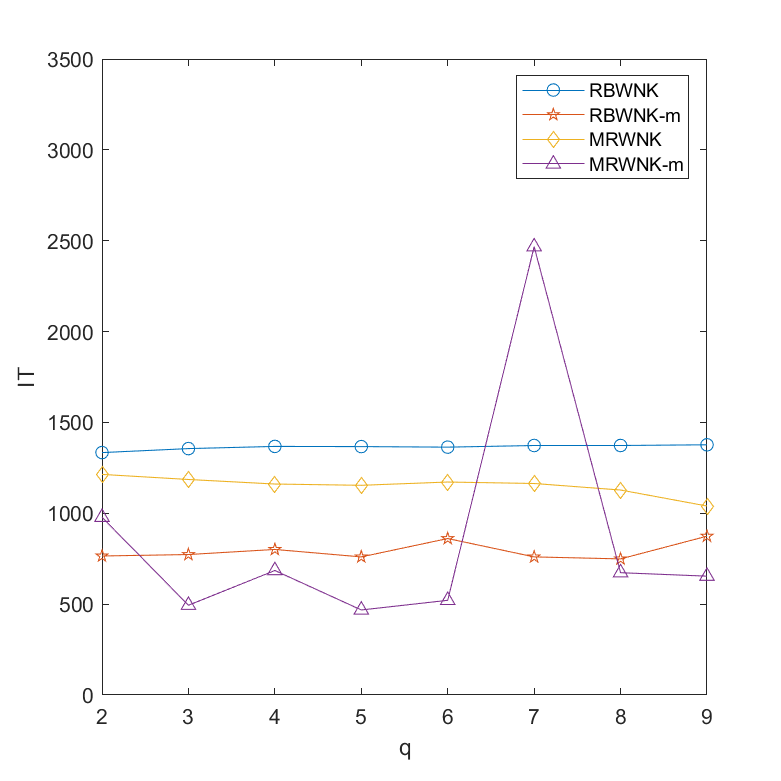}}
		\subfigure[$n=400$]{
			\includegraphics[scale=0.35]{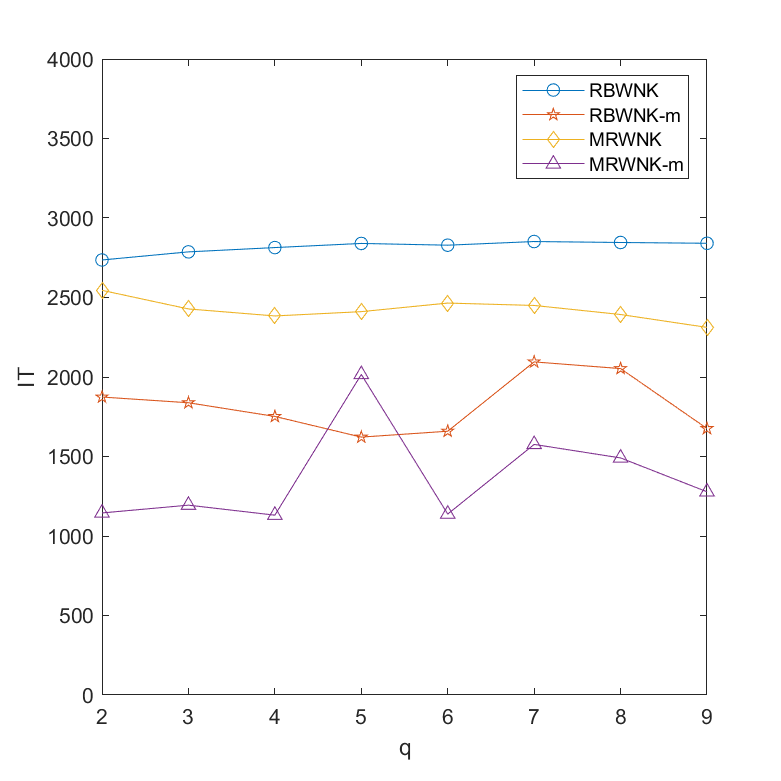}}
		\subfigure[$n=800$]{	\includegraphics[scale=0.35]{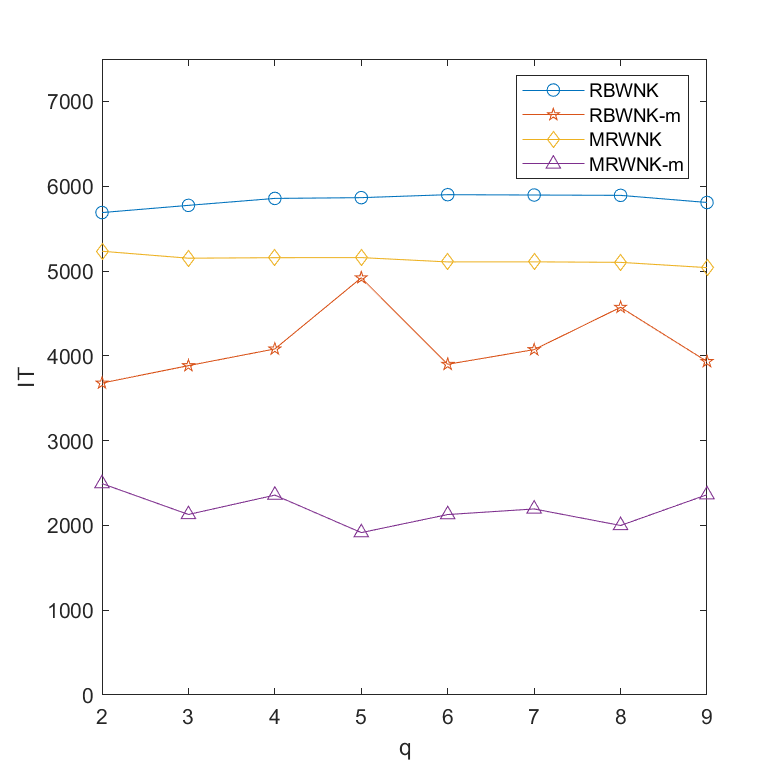}}\\
		\caption[d]{IT of  the four methods versus $q$ for nondquar problem with $m=n=200$, $400$, $800$}
		\label{figure nondquar q}
	\end{figure}
	
	\section{Conclusion}
	\label{sec5}
	
	Based on the RBWNK method and the greedy criterion of the MRABK method, we proposed two block Kaczmarz with momentum method: the RBWNK-m method and the MRWNK-m method. We note that RBWNK with $q=2$ is the NGABK and MRWNK-m with $q =2$, $\omega=0$ is the MRABK method. Under relatively mild conditions, the convergence rates of the new methods are analyzed. From numerical experiments, it is easy to see that MRWNK-m requires less IT and CPU compared to RBWNK, RBWNK-m, and MRWNK, and RBWNK-m requires less IT and CPU compared to RBWNK. Thus, these numerical experiments verified the efficiency of RBWNK-m and MRWNK-m. Meanwhile, we note that appropriate parameters $q$, $\omega$ and $\rho$ can significantly accelerate the convergence speed. In the future work, we will investigate how to obtain the optimal values of parameters.
	
\section*{Use of AI tools declaration}
The authors declare they have not used Artificial Intelligence (AI) tools in the creation of this article.

\section*{Conflict of interest}
The authors declare there is no conflict of interest.
\section*{Funding}
	 This work was supported by the National Natural Science Foundation of China [grant numbers
	42176011, 42374156], and the Fundamental Research Funds for the Central Universities of China [grant
	numbers 24CX03001A]

		\section*{Appendix}
	\begin{appendices}
		\section{The proof of E in Lemma \ref{lem}}\label{appendix A}
		\begin{align*}
			E
			&=-4\frac{\eta_k^Tf_{\tau_k}(x_k)}{\|f^{\prime}_{\tau_k}(x_k)^T\eta_k\|^2}\eta_k^Tf_{\tau_k}^{\prime}(x_k)(x_k-x^\dagger)\\
			&=-4\frac{\eta_k^Tf_{\tau_k}(x_k)}{\|f_{\tau_k}^{\prime}(x_k)^T\eta_k\|^2}\left(\sum_{i\in\tau_k}\eta_k^{(i)}f_i^\prime(x_k)\right)(x_k-x^\dagger)\\
			&=4\frac{\eta_k^Tf_{\tau_k}(x_k)}{\|f_{\tau_k}^{\prime}(x_k)^T\eta_k\|^2}\sum_{i\in\tau_k}\eta_k^{(i)}\left(f_i(x_k)-f_i(x^\dagger)-f_i^\prime(x_k)(x_k-x^\dagger)\right)-4\frac{\eta_k^T f_{\tau_k}(x_k)}{\|f_{\tau_k}^{\prime}(x_k)^T\eta_k\|^2}\sum_{i\in\tau_k}\eta_k^{(i)}f_i(x_k)\\
			&=F-4\frac{|\eta_k^Tf_{\tau_k}(x_k)|^2}{\|f_{\tau_k}^{\prime}(x_k)^T\eta_k\|^2}.
		\end{align*}
		When $q$ is even,
		\begin{align*}
			F
			&=4\frac{\eta_k^Tf_{\tau_k}(x_k)}{\|f_{\tau_k}^{\prime}(x_k)^T\eta_k\|^2}\sum_{i\in\tau_k}\eta_k^{(i)}\left(f_i(x_k)-f_i(x^\dagger)-f_i^\prime(x_k)(x_k-x^\dagger)\right)\\
			&=4\frac{\sum_{i\in\tau_k}f_i^{q-1}(x_k)f_i(x_k)}{\|f_{\tau_k}^{\prime}(x_k)^T\eta_k\|^2}\sum_{i\in\tau_k}f_i^{q-1}(x_k)\left(f_i(x_k)-f_i(x^\dagger)-f_i^\prime(x_k)(x_k-x^\dagger)\right)\\
			&\leq4\frac{\sum_{i\in\tau_k}f_i^{q}(x_k)}{\|f_{\tau_k}^{\prime}(x_k)^T\eta_k\|^2}\sum_{i\in\tau_k}\xi|f_i^{q-1}(x_k)||f_i(x_k)|\\
			&=4\xi\frac{\left(\sum_{i\in\tau_k}f_i^q(x_k)\right)^2}{\|f_{\tau_k}^{\prime}(x_k)^T\eta_k\|^2}\\
			&=4\xi\frac{\|f_{\tau_k}(x_k)\|_q^{2q}}{\|f_{\tau_k}^{\prime}(x_k)^T\eta_k\|^2}.
		\end{align*}
		The inequality makes use of Lemma \ref{qiexiangzhui}.\\
		When $q$ is odd,
		\begin{align*}
			F
			&=4\frac{\eta_k^Tf_{\tau_k}(x_k)}{\|f_{\tau_k}^{\prime}(x_k)^T\eta_k\|^2}\sum_{i\in\tau_k}\eta_k^{(i)}\left(f_i(x_k)-f_i(x^\dagger)-f_i^\prime(x_k)(x_k-x^\dagger)\right)\\
			&=4\frac{\sum_{i\in\tau_k}|f_i(x_k)|f_i^{q-2}(x_k)f_i(x_k)}{\|f_{\tau_k}^{\prime}(x_k)^T\eta_k\|^2}\sum_{i\in\tau_k}|f_i(x_k)|f_i^{q-2}(x_k)\left(f_i(x_k)-f_i(x^\dagger)-f_i^\prime(x_k)(x_k-x^\dagger)\right)\\
			&\leq4\frac{\sum_{i\in\tau_k}|f_i(x_k)|^q}{\|f_{\tau_k}^{\prime}(x_k)^T\eta_k\|^2}\sum_{i\in\tau_k}\xi|f_i(x_k)||f_i(x_k)^{q-2}||f_i(x_k)|\\
			&=4\xi\frac{\left(\sum_{i\in\tau_k}|f_i(x_k)|^q\right)^2}{\|f_{\tau_k}^{\prime}(x_k)^T\eta_k\|^2}\\
			&=4\xi\frac{\|f_{\tau_k}(x_k)\|_q^{2q}}{\|f^{\prime}(x_k)^T\eta_k\|^2}.
		\end{align*}
		The inequality makes use of \ref{qiexiangzhui}. \\
		 So $F\leq4\xi\frac{\|f_{\tau_k}(x_k)\|_q^{2q}}{\|f_{\tau_k}^{\prime}(x_k)^T\eta_k\|^2}$. And $|\eta_k^Tf_{\tau_k}(x_k)|^2=(\sum_{i\in\tau_k}|f_i(x_k)|^q)^2=\|f_{\tau_k}(x_k)\|_q^{2q}$. Then
		\begin{equation*}
			E\leq-4(1-\xi)\frac{\|f_{\tau_k}(x_k)\|_q^{2q}}{\|f_{\tau_k}^{\prime}(x_k)^T\eta_k\|^2}.
		\end{equation*}
		\section{The proof of \texorpdfstring{$\frac{\|f_{\tau_k}(x_k)\|_q^{2q}}{\|f_{\tau_k}^{\prime}(x_k)^T\eta_k\|^2}$}. in Lemma \ref{lem}}\label{appendix B}
		Since Assumption \ref{asm1} (i) is satisfied, then with the use of the properties of continuous functions on a bounded closed domain, there exists $\alpha>0$, such that $\alpha=
		\underset{1\leq i\leq \tau}{max}\underset{x\in  \mathcal{D}}{sup}\|f'_{i}(x)\|^2.$
    	Thus, we have
		\begin{equation}\label{lem1_frac}
			\frac{\|f_{\tau_k}(x_k)\|_q^{2q}}{\|f_{\tau_k}^{\prime}(x_k)^T\eta_k\|^2}\geq \frac{\|f_{\tau_k}(x_k)\|_q^{2q}}{\|f^{\prime}_{\tau_k}(x_k))\|_F^2\|\eta_k\|^2}\geq \frac{\|f_{\tau_k}(x_k)\|_q^{2q}}{|\tau_k|\alpha\|\eta_k\|^2}.	
		\end{equation}
		The last equation makes use of $\|\eta_k\|^2=\|f_{\tau_k}(x_k)\|_{2q-2}^{2q-2}$, which can be find in \cite{Ye24}.
		Using H\"{o}lder's inequality, we have
		\begin{equation*}
			\|f_{\tau_k}(x_k)\|_2^2\leq\|f_{\tau_k}(x_k)\|_q^2|\tau_k|^{\frac{q-2}{q}},
		\end{equation*}
		where $|\tau_k|$ is the cardinality of the selected block indices $\tau_k$.
		Thus we have
		\begin{equation}\label{holder}
			\|f_{\tau_k}(x_k)\|_q^{2q}\geq\|f_{\tau_k}(x_k)\|_2^{2q}|\tau_k|^{2-q}.
		\end{equation}
		What's more,
		\begin{equation}\label{lem1_2q-2}
			\|f_{\tau_k}(x_k)\|_{2q-2}^{2q-2}=\sum_{i\in\tau_k}f_i^{2q-2}(x_k)\leq\left(\sum_{i\in\tau_k}f_i^2(x_k)\right)^{q-1}=\|f_{\tau_k}(x_k)\|_2^{2q-2}.
		\end{equation}
		Then substituting (\ref{lem1_2q-2}) and (\ref{holder}) into (\ref{lem1_frac}), one can get that
		\begin{equation*}
			\frac{\|f_{\tau_k}(x_k)\|_q^{2q}}{\|f_{\tau_k}^{\prime}(x_k)^T\eta_k\|^2}\geq\frac{\|f_{\tau_k}(x_k)\|_2^{2q}|\tau_k|^{2-q}}{|\tau_k|\alpha\|f_{\tau_k}(x_k)\|_{2}^{2q-2}}=\frac{|\tau_k|^{1-q}}{\alpha}\|f_{\tau_k}(x_k)\|_2^2.
		\end{equation*}
		
	\end{appendices}

	\bibliographystyle{model1-num-names}
	\bibliography{reference1}
	
\end{document}